\documentclass[11pt,reqno,oneside]{amsart}
\usepackage{amscd}
\usepackage{amsmath}
\usepackage{latexsym}
\usepackage{amsfonts}
\usepackage{amssymb}
\usepackage{amsthm}
\usepackage{graphicx}
\usepackage{verbatim}
\usepackage{mathrsfs}
\usepackage{enumerate}
\usepackage[hypertexnames=false]{hyperref}
\usepackage{fullpage}
\usepackage{xcolor}

\theoremstyle{plain}
\theoremstyle{plain}\newtheorem{theorem}{Theorem}[section]
\theoremstyle{plain}\newtheorem{lemma}[theorem]{Lemma}
\theoremstyle{plain}
\theoremstyle{plain}
\theoremstyle{plain}\newtheorem{remark}{Remark}[section]

\newcommand{\B}{\Big}

\newcommand{\be}{\begin{equation}}
\newcommand{\ee}{\end{equation}}
\newcommand{\ba}{\begin{aligned}}
	\newcommand{\ea}{\end{aligned}}

\newcommand{\f}{\frac}

\newcommand{\ben}{\begin{enumerate}}
	\newcommand{\een}{\end{enumerate}}

\newcommand{\Rmnum}[1]{\expandafter\@slowromancap\romannumeral #1@}

\numberwithin{equation}{section}
\allowdisplaybreaks
\begin{document}
	\title{A new blow-up criterion for the 2D full compressible Navier-Stokes equations without heat conduction in a bounded domain  }
\author[J. Fan, Q. Jiu]{ Jie Fan$^{1}$, Quansen Jiu$^{2}$}
\address{$^1$ School of Mathematical Sciences, Capital Normal University, Beijing, 100048, P.R.China}

\email{fj0828@outlook.com}
\address{$^1$ School of Mathematical Sciences, Capital Normal University, Beijing, 100048, P.R.China}

\email{jiuqs@cnu.edu.cn}
	
\date{}

\maketitle
\section*{abstract}
 This paper is to derive a new blow-up criterion for the 2D full compressible Navier-Stokes equations without heat conduction in terms of the
 density $\rho$ and the pressure $P$. More precisely, it indicates that in a bounded domain the strong solution exists globally if the norm
 $\|\rho||_{{L^\infty(0,t;L^{\infty})}}+||P||_{L^{p_0}(0,t;L^\infty)}<\infty$ for some constant $p_0$ satisfying $1<p_0\leq 2$. The boundary condition is imposed as a Navier-slip boundary one and the initial vacuum is permitted. Our result extends previous one which is stated as $\|\rho||_{{L^\infty(0,t;L^{\infty})}}+||P||_{L^{\infty}(0,t;L^\infty)}<\infty$.
		\section{Introduction}
	\label{intro}
	\setcounter{section}{1}\setcounter{equation}{0}
The full compressible Navier-Stokes equations  in $\Omega\subseteq\mathbb{R}^n$~($n=2,3)$~ bounded or unbounded read as:
 \begin{equation}\label{FNS}
\begin{cases}
	&\rho_t+\nabla \cdot (\rho u)=0, \\
&\rho  u_t+\rho u\cdot\nabla u+\nabla
P -\mu\Delta u-(\mu+\lambda)\nabla\text{div\,}u=0,\\
&c_{v}[\rho \theta_t+\rho u\cdot\nabla\theta] +P\text{div\,}u-\kappa\Delta\theta=\frac{\mu}{2}\left|\nabla u+(\nabla
u)^{\text{tr}}\right|^2+\lambda(\text{div\,}u)^2,\\
\end{cases}\end{equation}
with the following initial data
$$(\rho, u, \theta)=(\rho_0, u_0, \theta_0) \quad\text{in}~~\Omega, $$
and the  Dirichlet boundary condition
\begin{equation}\begin{aligned}\label{Dirichlet}
&u=0,~~\theta=0 \quad\text{on}~~\partial\Omega,
\end{aligned}\end{equation}
or the Neumann boundary condition on $\theta$,
\begin{equation}\begin{aligned}\label{Neumann}
&u=0,~~\frac{\partial\theta}{\partial n}=0, ~\quad\text{on}~\partial\Omega.
\end{aligned}\end{equation}
Here, $x=(x_1,...,x_n)\in\Omega$~is the spatial coordinate and ~$t> 0$~is the time,~$
\rho=\rho(x,t)$~and~$u=(u_1(x,t),...,u_n(x,t))$~stand for the flow density and velocity respectively, and the pressure $P$ is determined by
\begin{equation}\begin{aligned}\label{667}
P=R\rho\theta,R>0.
\end{aligned}\end{equation}
The positive constant $c_{v}$ and $\kappa$ are respectively the heat capacity, the thermal conductivity coefficient. $\mu$ and $\lambda$ are the coefficients of viscosity, which are assumed to be constants, satisfying the following
 physical restrictions:
\begin{equation}\label{nares}
\mu>0,\ 2\mu+n\lambda\ge0.\end{equation}
There are a large number of literatures about the well-posedness  to the compressible Navier-Stokes system \eqref{FNS}. Matsumura-Nishida \cite{[MN]} first obtained the global classical solution when initial data is close to a non-vacuum equilibrium and the solution has
small perturbations from a uniform nonvacuum state. Later, in the absence of vacuum the global existence of weak solutions for discontinuous initial data
was studied by~Hoff \cite{[DFD]}. For isentropic compressible Navier-Stokes equations, Lions \cite{[Lions]} obtained the existence of global weak solutions with ~$\gamma\geq\frac{3N}{N+2}$~for~$N=2, 3.$~ Feireisl \cite{[Feireisl Novotny]} extended Lion's result with~$\gamma>\frac{3}{2}$~and Jiang-Zhang \cite{[JPO]} relaxed the restriction for~$\gamma>1$~with spherical symmetric initial data. For the initial data far away from vacuum, Danchin \cite{[Danchin]} established the local well-posedness in critical
spaces. If vacuum is taken into account, the local existence of a strong solution with initial data containing  the compatibility condition was established by Cho and Kim
\cite{[CK]}. For the Cauchy problem or for the initial-boundary value problem of 3D and 2D barotropic compressible Navier-Stokes equations,
Huang-Li-Xin\cite{[HLX small energy]}, Li-Xin\cite{[LX2D small]} and  Cai-Li\cite{[JM]} established the global existence of classical solutions for initial data with small
energy but possibly large oscillations and containing vacuum states, respectively. Xin \cite{[Xin]} (see \cite{[li wang xin]},\cite{[JWX]},\cite{[XY]} for further states) first proved
that full compressible Navier-Stokes equations with $\kappa=0$ will blow up in finite time if the initial density has compact support. For this reason, many studies are devoted to
 the mechanism of blow up and structure of possible singularities of strong  solutions to the compressible Navier-Stokes equations.
More precisely, suppose that $0<T^*<\infty$ is the maximal time of existence of a strong solution of system \eqref{FNS}. Then when ~$\kappa>0,$~ kinds of   blow-up criteria have been established, of which some are listed as follows:

\noindent Fan-Jiang-Ou  \cite{[FJO]}:
$$
\lim\sup\limits_{t\rightarrow
T^*}\left(\|\nabla
u\|_{L^1(0,t;L^\infty(\mathbb{R}^{3}))}+\|\theta\|_{L^\infty(0,t;L^\infty(\mathbb{R}^{3}))}\right)=\infty,~~~(\lambda< 7\mu);
$$
Sun-Wang-Zhang \cite{[SWZarma]}:
$$
\lim\sup\limits_{t\rightarrow
T^*}\left( \|(\rho,\f{1}{\rho},\theta)\|_{L^\infty(0,t;L^\infty(\mathbb{R}^{3}))}\right)=\infty,~~~(\lambda< 7\mu);
$$
Jiu-Wang-Ye \cite{[JYWR]}:
\begin{align}
&\limsup_{t\rightarrow T^*}
  \B(\|\rho\|_{L^{\infty}(0,t;L^{\infty}(\mathbb{R}^{3}))}+\|\theta\|_{L^{p}(0,t;L^{q}(\mathbb{R}^{3}))}\B)= \infty,\label{JWY1}
  ~\text{with}~~\f{2}{p}+\f{3}{q}=2,\ \  q>\f{3}{2}~~ (\lambda<3\mu);\\
  &\limsup_{t\rightarrow T^*}\B(
  \|\rho\|_{L^{\infty}(0,t;L^{\infty}(\mathbb{R}^{3}))}+ \|\text{div\,}u \|_{L^{2}(0,t;L^{3}(\mathbb{R}^{3}))}
  +\|\theta\|_{L^{p}(0,t;L^{q}(\mathbb{R}^{3}))}\B)= \infty,\\ \label{JWY2}
  ~&\text{with}~~\f{2}{p}+\f{3}{q}=2,\ \  q>\f{3}{2};\nonumber
  \end{align}
Fang-Zi-Zhang \cite{[FZZ]}:
\be\label{FZZ}
\lim\sup\limits_{t\rightarrow
T^*}\B(\|\rho\|_{L^\infty(0,t;L^\infty(\Omega))}+\|\theta\|_{L^\infty(0,t;L^\infty(\Omega))}\B)
=\infty
\ee
for smooth bounded domain $\Omega$ in $\mathbb{R}^{2}$ and boundary condition \eqref{Dirichlet}.

\noindent Fan-Jiu-Wang \cite{[FJWX]}:
$$
\lim\sup\limits_{t\rightarrow
T^*}\B( \|\rho\|_{L^\infty(0,t;L^\infty(\Omega))}+\|\theta||_{L^p(0,t;L^q(\Omega))}\B)=\infty,~\text{with}~~\frac{1}{p}+\frac{1}{q}=1,\ \
2\leq q\leq\infty
$$
for smooth bounded domain $\Omega$ in $\mathbb{R}^{2}$ and  boundary condition \eqref{Neumann}.

\noindent Feireisl-Wen-Zhu  \cite{[Feireisl WENZHU]} removed the technical conditions in \cite{[JYWR]} relating the values of the shear and bulk viscosity coefficients, and proved
\begin{align}
  &\limsup_{t\rightarrow T^*}\B(
  \|\rho\|_{L^{\infty}(0,t;L^{\infty}(\Omega))}+\|\theta\|_{L^{p}(0,t;L^{q}(\Omega))}\B)= \infty,
  ~\text{with}~~\f{2}{p}+\f{3}{q}=2,\ \  q\in(\f{3}{2},\infty]\label{JWY2}
  \end{align}
 for~$\mathbb{R}^3$~or boundary condition \eqref{Neumann}.

This paper is denoted to establishing some blow-up criteria for the 2D full compressible Navier-Stokes equations without heat conduction ($\kappa=0$) in a bounded domain. Let $\Omega$ be a bounded smooth domain in $\mathbb{R}^2$. When~$\kappa=0$ and take~$c_v=R=1$ for simplicity, the system~\eqref{FNS}~can
be rewritten as
 \begin{equation}\label{FNSZ}
\begin{cases}
	&\rho_t+\nabla \cdot (\rho u)=0, \\
&\rho  u_t+\rho u\cdot\nabla u+\nabla
P -\mu\Delta u-(\mu+\lambda)\nabla\text{div\,}u=0,\\
&P_t+\text{div}(Pu)+P\text{div}u=\frac{\mu}{2}\left|\nabla u+(\nabla u)^{\text{tr}}\right|^2+\lambda(\text{div}u)^2.
	\end{cases}\end{equation}
Here $\mu$~and~$\lambda$~are the coefficients of viscosity, which are assumed to be constants, satisfying the following physical restrictions:
\begin{equation}\begin{aligned}\label{668}
\mu>0,\,\, \mu+\lambda\geq 0.
\end{aligned}\end{equation}
System \eqref{FNSZ}-\eqref{668} is supplemented with the initial data
\begin{equation}\begin{aligned}\label{669}
\rho(x,0)=\rho_0(x),\quad{\rho u(x,0)=m_0(x),} \quad{P(x,0)=P_0,}~~x\in\Omega,
\end{aligned}\end{equation}
and Navier-slip boundary conditions:
\begin{equation}\begin{aligned}\label{666}
&u\cdot{n}=0, \quad\text{and} \quad{\text{curl}u=0} \quad\text {on}\quad{\partial\Omega,}
\end{aligned}\end{equation}
where~$n=(n_1,n_2)$~is the unit outer normal vector to~$\partial\Omega.$~

 Suppose that $0<T^*<\infty$ is the maximal time for the existence of a strong  solution to \eqref{FNSZ}.
 Huang \cite{Huangxinwithoutzero} established a blowup criterion in $\mathbb{R}^3$ or in a bounded domain with the Dirichlet boundary condition for the velocity:
$$
\lim\sup\limits_{t\rightarrow
T^*}\B(\|\rho\|_{L^\infty(0,t;L^\infty(\Omega))}+\|\theta||_{L^\infty(0,t;L^\infty(\Omega))}\B)=\infty,~~~(\mu>4\lambda).
$$
Zhong \cite{[ZXR2],[ZXC]} obtained a blowup criterion in $\mathbb{R}^2$ or in a bounded domain $\Omega$ of $\mathbb{R}^2$ with the Dirichlet boundary condition for the velocity:
\begin{equation}\begin{aligned}\label{3092}
\lim\sup\limits_{t\rightarrow
T^*}\B(\|\rho\|_{L^\infty(0,t;L^\infty(\Omega))}+||P||_{L^\infty(0,t;L^\infty(\Omega))}\B)=\infty.
\end{aligned}\end{equation}

\noindent Wang \cite{[WYF]} proved that
$$
\lim\sup\limits_{t\rightarrow
T^*}\B(\|\rho||_{{L^\infty(0,t;L^{\infty}(\mathbb{R}^2)}}+\|\theta||_{L^2(0,t;L^\infty(\mathbb{R}^2))}\B)=\infty.
$$

After some modifications, the local existence and uniqueness of the strong solution to \eqref{FNSZ}-\eqref{666} with initial data containing vacuum can be established  as in  \cite{[CK]}, which can be stated as follows.
\begin{theorem}\label{LSC}    Let $q\in(2,\infty).$ Assume that
the initial data satisfies
$$
\rho_0\in W^{1,q},\,\, u_0\in D^1\cap D^2,\,\, P_0\in W^{1,q},
$$
and the following compatibility condition:
\begin{equation}\begin{aligned}\label{392}
-\mu\triangle u_0-(\mu+\lambda)\nabla\text{div}u_0+\nabla P=\sqrt{\rho}_0g,
\end{aligned}\end{equation}
for some~$g\in L^2$. Then there exist a positive $T_0$ and a unique strong solution $(\rho, u, P)$ to \eqref{FNSZ}-\eqref{666} such that,
\begin{equation}\label{FNSS}
\begin{cases}
	&\rho\geq0,\rho\in C([0,T_0];W^{1,q}),\,\, \rho_t\in C([0,T_0];L^q), \\
& u\in C([0,T_0];D^1\cap D^2)\cap L^2(0,T_0;D^{2,q}),\\
&\sqrt{\rho}u,\,\,\sqrt{\rho}\dot{u}\in L^{\infty}(0,T_0;L^2),\\
&P\geq 0,\,\, P\in C([0,T_0];W^{1,q}),\,\, P_t\in C([0,T_0];L^q).
	\end{cases}\end{equation}
\end{theorem}
Then our main results can be stated as:
\begin{theorem}\label{WAW}Suppose that $(\rho, u, P)$ is the unique strong solution in Theorem \ref{LSC} to \eqref{FNSZ}-\eqref{666}. If the maximal
existence time~$T^{\ast}$~is finite, then there holds
\begin{equation}\begin{aligned}\label{302}
\lim\sup\limits_{t\rightarrow
T^*}\left(\|\rho||_{{L^\infty(0,t;L^{\infty})}}+||P||_{L^{p_0}(0,t;L^\infty)}\right)=\infty,
\end{aligned}\end{equation}
where~$p_0$~is some constant satisfying~$1<p_0\leq2.$~
\end{theorem}
\begin{remark}
In comparision with \cite{Huangxinwithoutzero} in which a blow-up criterion in terms of the upper bound of the temperature was obtained under the condition~$\mu>4\lambda,$~ there is no need to impose the technical conditions relating the values of the  viscosity coefficients ~$\mu, \lambda$~as well as we establish a blow-up criterion in terms of the integrability of the pressure. This result gives an answer of Nash's conjecture proposed in \cite{[Nash]} without thermal conductivity, which says:
\par {\itshape  "Probably one should first try to prove a conditional existence and uniqueness theorem for flow equations. This should give existence, smoothness, and continuation (in time) of flows, conditional on the non-appearance of
certain gross type of singularity, such as infinities of temperature or density."}
\end{remark}
\begin{remark}Due to the fact that~$||P||_{L^{p_0}(0,t;L^\infty)}\leq||P||_{L^{\infty}(0,t;L^\infty)},$ we extend Zhong's result in \cite{[ZXR2],[ZXC]} under Navier-slip conditions.
\end{remark}

To prove our main result, we apply for the effective viscous flux $F$, and the vorticity $w$ are defined as follows:
\begin{equation}\begin{aligned}\label{11.1}
F\triangleq(2\mu+\lambda)\text{div}u-P,\quad{w\triangleq \text{curl} u=\nabla^{\perp}\cdot{u}=\partial_{x_2}u_1-\partial_{x_1}u_2}.
\end{aligned}\end{equation}

The main result is proved by a contradiction argument. Suppose that \eqref{302} was false. Then we first obtain the uniform estimate of $\|\nabla u||_{{L^2(0,T; L^2)}}.$ The second step is to estimate $\|\nabla u||_{{L^\infty(0,T; L^2)}}.$ In
 order to get the $L^{\infty}_T L^2_x$ norm of $\sqrt{\rho}\dot{u}$ (in Lemma \ref{BBA}), we need to deal with the term $\|\nabla u||_{L^\infty L^4},$ where $\dot{u}=u_t+u\cdot{\nabla
u}$~represents the material derivative. And we also need the estimate of $||P||_{L^\infty L^4}.$ The effective viscous flux $F$ and the vorticity $w$ solve the Neumann problem and
 the related Dirichlet problem respectively. We use the standard $L^p$ theory to estimate $\|\nabla F||_{L^\infty L^p}$ and $\|\nabla w||_{L^\infty L^p}$ which plays a
 crucial role in estimating $\|\nabla u||_{L^\infty L^p}(p=3, p=4).$ The main difficulty in carrying out this construction is that integrals on the boundary
 $\partial\Omega,$
which are $\int_{\partial\Omega}F_t(\dot{u}\cdot{n})ds$ and $\int_{\partial\Omega}u\cdot{\nabla F}(\dot{u}\cdot{n})ds,$ appear and are required to be dealt with. To overcome this difficulty, as in Cai-Li \cite{[JM]}, together with the slip boundary condition~$u\cdot{n}|_{\partial \Omega}$~, we have
\begin{equation}\begin{aligned}\label{13.1}
u=(u\cdot{n^{\perp}})n^{\perp},(u\cdot{\nabla})u\cdot{n}=-(u\cdot{\nabla})n\cdot{u},
\end{aligned}\end{equation}
where $n^{\perp}$ is the unit tangential vector on the boundary $\partial \Omega$ denoted by
\begin{equation}\begin{aligned}\nonumber
n^{\perp}\triangleq(n_2,-n_1).
\end{aligned}\end{equation}
Furthermore, using Green formula yields that for~$f\in H^1,$~
\begin{equation}\begin{aligned}\label{13.3}
\int_{\partial \Omega}u\cdot{\nabla f}ds=&\int_{\partial \Omega}(u\cdot{n^{\perp}})n^{\perp}\cdot{\nabla f}ds\\
=&\int_{\Omega}\nabla f\cdot{\nabla^{\perp}}(u\cdot{n^{\perp}})dx\leq C||f||_{H^1}||u||_{H^1},
\end{aligned}\end{equation}
where $\nabla^{\perp}f=(\partial_2,-\partial_1)f.$

Making full use of \eqref{13.1} and \eqref{13.3}, we can rewrite
\begin{equation}\begin{aligned}
&\int_{\partial\Omega}F((u\cdot{\nabla u})\cdot{\nabla n}\cdot{u})ds\\
&=\int_{\partial\Omega}F(u\cdot{n^{\perp}})n^{\perp}\cdot{\nabla u_i\partial_i n_ju_j}ds\\
&=\int_{\Omega}\nabla^{\perp}\cdot{\left(\nabla u_{i}\partial_{i}n_{j}u_{j}F(u\cdot{n^{\perp}})\right)}dx\\
&=\int_{\Omega}\nabla u_{i}\cdot{\nabla^{\perp}\left(\partial_{i}n_{j}u_{j}F(u\cdot{n^{\perp}})\right)}dx,\\
\end{aligned}\end{equation}
and
\begin{equation}\begin{aligned}
\int_{\partial\Omega}(u\cdot{\nabla F})(\dot{u}\cdot{n})ds=&\int_{\partial\Omega}(u\cdot{n^{\perp}})n^{\perp}\cdot{\nabla F(\dot{u}\cdot{n})}ds\\
=&\int_{\Omega}\nabla^{\perp}\cdot(({u\cdot{n^{\perp}})\nabla F(\dot{u}\cdot{n})})dx\\
=&\int_{\Omega}\nabla F\cdot{\nabla^{\perp}((u\cdot{n^{\perp}})(\dot{u}\cdot{n}))}dx.\\
\end{aligned}\end{equation}

According to local well-posedness in Theorem \ref {LSC}, we need to estimate $\|\nabla\rho||_{{L^\infty}(0,T;L^p)}$ and $\|\nabla P||_{{L^\infty}(0,T;L^p)}$ by solving a logarithmic Gronwall inequality, where~$p>2.$~
$\|\nabla\rho||_{{L^\infty}(0,T;L^p)}$ can be determined by $\|\nabla u||_{L^1(0,T;L^{\infty})}$ due to the scalar hyperbolic structure of
$\eqref{FNS}_{1}.$
For $\|\nabla P||_{L^\infty L^p}$, we have
\begin{equation}\begin{aligned}\label{190}
&\frac{d}{dt}\|\nabla P||_{L^p}\leq C(||P||_{L^\infty}+\|\nabla u||_{L^\infty})(\|\nabla P||_{L^p}+\|\nabla^2 u||_{L^p}).
\end{aligned}\end{equation}
Thus, in order to control the norm of $\|\nabla P||_{{L^\infty}(0,T;L^p)}$, it suffices to get the upper bound of $\|\nabla
u||_{L^\infty}\|\nabla^2 u||_{L^p}.$
Indeed, we will obtain
\begin{equation}\begin{aligned}\label{11.17}
\|\nabla^2 u||_{L^p}\leq&C(\|\nabla\dot{u}||^{\frac{p(p-2)}{p^2-2}}_{L^2}+\|\nabla P||_{L^p}),
\end{aligned}\end{equation}
and
\begin{equation}\begin{aligned}
\|\nabla u||_{L^\infty}\leq&C\left(1+\|\nabla\dot{u}||^{\frac{p^2(p-2)}{2(p-1)(p^2-2)}}_{L^2}+||P||_{L^\infty}\right)\log(e+\|\nabla^2
u||_{L^p})+C\|\nabla u||_{L^2}+C.\\
\end{aligned}\end{equation}
where $\frac{p(p-2)}{p^2-2},\frac{p^2(p-2)}{2(p-1)(p^2-2)}\in(0,1).$
\par
Then it yields
\begin{equation}\begin{aligned}\label{12.350}
&\|\nabla u||_{L^\infty}\|\nabla^2 u||_{L^p}\\
\leq&C\|\nabla\dot{u}||^{\frac{p(p-2)}{p^2-2}}_{L^2}\left(1+\|\nabla\dot{u}||^{\frac{p^2(p-2)}{2(p-1)(p^2-2)}}_{L^2}+||P||_{L^\infty}\right)\log(e+\|\nabla^2
u||_{L^p})+\quad\text{other}\quad\text{terms}\\
\leq& C\|\nabla\dot{u}||^2_{L^2}+\quad\text{other}\quad\text{terms}.
\end{aligned}\end{equation}
\par
It should be remarked  that in the process of our proof, we just used $||P||_{L^1(0,T;L^{\infty})}$ during estimating the low regularity of the velocity. When we get the improved regularity of the density and pressure, we need $||P||_{L^{p_0}(0,T;L^{\infty})}$,where~$p_0$~is some constant satisfying~$1<p_0\leq2.$~
Putting \eqref{12.350} into \eqref{190} leads to
\begin{equation}\begin{aligned}
&\frac{d}{dt}\|\nabla P||_{L^p}\\
\leq&C\left(1+\|\nabla\dot{u}||^{\frac{p^2(p-2)}{2(p-1)(p^2-2)}}_{L^2}+||P||_{L^\infty}\right)\log\left(e+\|\nabla\dot{u}||^{\frac{p(p-2)}{p^2-2}}_{L^2}+\|\nabla
P||_{L^p}\right)\left(\|\nabla\dot{u}||^{\frac{p(p-2)}{p^2-2}}_{L^2}+\|\nabla P||_{L^p}+1\right).
\end{aligned}\end{equation}
Together with Young inequality we have,
\begin{equation}\begin{aligned}
&\frac{d}{dt}\|\nabla P||_{L^p}\leq C||P||^{p_0}_{L^\infty}+\|\nabla\dot{u}||^2_{L^2}+\quad\text{other}\quad\text{terms},
\end{aligned}\end{equation}
where~$p_0$~is some constant satisfying~$1<p_0\leq2$,~which leads to the estimates of~$\|\nabla P||_{L^\infty}(0,T;L^p).$~
\par
Some notations are introduced. For $1\leq p\leq\infty$, $L^{p}(\Omega)$ represents the usual
 Lebesgue  spaces. For a nonnegative integer $k$ the classical Sobolev space  $W^{k,p}(\Omega)$
is equipped with the norm $\|f\|_{W^{k,p}(\Omega)}=\sum\limits_{|\alpha| =0}^{k}\|D^{\alpha}f\|_{L^{p}(\Omega)},$ with $\alpha=(\alpha_1,\alpha_2).$
A function $f$ belongs to  the homogeneous Sobolev spaces $D^{k,l}$
if $
u\in L^1_{\rm{loc}}(\Omega)$ and $\|\nabla^k u \|_{L^l}<\infty.$ For simplicity,   we write
$$L^p=L^p(\Omega),   \ H^k=W^{k,2}(\Omega), \ D^k=D^{k,2}(\Omega).$$

  The paper is organized as follows. In section 2, we will derive the regularity of the velocity and pressure. In section 3, we will prove the main result
  Theorem \ref{WAW}.

\section{Regularity of the velocity and pressure}
We will prove our main result by contradiction arguments.
Suppose that the assertion of the Theorem \ref{WAW} was false, namely
\begin{equation}\begin{aligned}\label{12.1}
\lim\sup\limits_{t\rightarrow
T^*}\left(\|\rho||_{{L^\infty(0,t;L^{\infty})}}+||P||_{L^{p_0}(0,t;L^\infty)}\right)\leq M,
\end{aligned}\end{equation}
where~$p_0$~is some constant satisfying~$1<p_0\leq2,$~and~$M$~is a finite number.
\par
The div-curl control in \cite{{[HG]},{[SY]},{[AG]}} will be used to get the estimate of~$\|\nabla u||_{L^p}$.~
\begin{lemma}\label{LAJ}Let ~$1<p<\infty$~and ~$\Omega$~be a bounded domain in ~$\mathbb{R}^2$~with Lipschitz boundary ~$\partial\Omega.$~For~$u\in
W^{1,q},$~if~$\Omega$~is simply connected and~$u\cdot{n}=0$~on~$\partial\Omega,$~then it holds that
\begin{equation}\begin{aligned}\label{11.3}
\|\nabla u||_{L^p}\leq C\left(\|\text{div}~u||_{L^p}+\|\text{curl}~u||_{L^p}\right).
\end{aligned}\end{equation}
\end{lemma}
The following is the standard energy estimate.
\begin{lemma}\label{QQW} Suppose that \eqref{12.1} is valid. Then there holds
\begin{equation}\begin{aligned}\label{10.3}
\sup_{0\leq t\leq T}\int_{\Omega}\left(\rho |u|^2+P\right)dx+\int_{0}^{T}\|\nabla u||^2_{L^2}dt\leq C,
\end{aligned}\end{equation}
where  C denotes generic positive constant depending only on $\Omega,$ $M,$ $\lambda,$ $\mu$, $T^*$
and the initial data for any ~$T\in[0,T^{\ast}).$~
\end{lemma}
\begin{proof}
The pressure~$P$ satisfies
\begin{equation}\begin{aligned}\label{18.2}
P_t+u\cdot{\nabla P}+2P\text{div}u=G\triangleq\frac{\mu}{2}\left|\nabla u+(\nabla u)^{\text{tr}}\right|^2+\lambda(\text{div}u)^2\geq0.
\end{aligned}\end{equation}
Define the particle path before blow-up time
\begin{equation}
\begin{cases}
	\frac{d}{dt}X(x,t)=u(X(x,t),t), \\
X(x,0)=x.
	\end{cases}\end{equation}
Along with the particle path, we obtain from~\eqref{18.2}~that
\begin{equation}\begin{aligned}\nonumber
\frac{d}{dt}P(X(x,t),t)=-2P\text{div}u+G,
\end{aligned}\end{equation}
which implies
\begin{equation}\begin{aligned}\nonumber
P(X(x,t),t)=\exp\left(-2\int_{0}^{t}\text{div}uds\right)\left[P_0+\int_{0}^{t}\exp\left(2\int_{0}^{s}\text{div}u d\tau\right)Gds\right]\geq 0.
\end{aligned}\end{equation}
Integrating \eqref{18.2} over~$\Omega$~and using slip boundary condition~\eqref{666},~we arrive at
\begin{equation}\begin{aligned}\label{1234}
\frac{d}{dt}\int_{\Omega}Pdx=&-\int_{\Omega}P\text{div}udx+\int_{\Omega}\left(\frac{\mu}{2}\left|\nabla u+(\nabla
u)^{\text{tr}}\right|^2+\lambda(\text{div}u)^2\right)dx\\
\leq&C||P||^2_{L^2}+C_0\|\nabla u||^2_{L^2}\\
\leq&C||P||_{L^\infty}||P||_{L^1}+C_0\|\nabla u||^2_{L^2},
\end{aligned}\end{equation}
where~$C_0$~is contant depending on~$\lambda, \mu.$~

In view of $\Delta u=\nabla\text{div}u+\nabla^{\perp}w,$ we rewrite the equation of conservation of momentum $\eqref{FNSZ}_{2}$ as
\begin{equation}\begin{aligned}\label{10.2}
\rho\dot{u}+\nabla P=(2\mu+\lambda)\nabla\text{div}u+\mu\nabla^{\perp}w.
\end{aligned}\end{equation}
Multiplying \eqref{10.2} by $u$ and integrating over $\Omega$, together with the boundary conition \eqref{666}, we see that
\begin{equation}\begin{aligned}\label{112}
\frac{1}{2}\frac{d}{dt}\int_{\Omega}\rho|u|^2dx+(2\mu+\lambda)\int_{\Omega}(\text{div}u)^2dx+\mu \int_{\Omega}w^2dx=&\int_{\Omega}P\text{div}udx\\
\leq& C||P||_{L^\infty}||P||_{L^1}+\frac{\epsilon}{4}\|\nabla u||^2_{L^2},
\end{aligned}\end{equation}
together with \eqref{11.3}, we obtain after choosing $\epsilon$ suitably small that
\begin{equation}\begin{aligned}\label{002}
\frac{d}{dt}\int_{\Omega}\rho|u|^2dx+\int_{\Omega}|\nabla u|^2dx\leq& C||P||_{L^\infty}||P||_{L^1}.
\end{aligned}\end{equation}
Multiplying \eqref{002} by $(C_0+1)$~and then adding to \eqref{1234}, one gets that
\begin{equation}\begin{aligned}\label{113}
\frac{d}{dt}\int_{\Omega}\left(P+(C_0+1)\rho|u|^2\right)dx+\int_{\Omega}|\nabla u|^2dx\leq C||P||_{L^\infty}||P||_{L^1}.
\end{aligned}\end{equation}
Consequently, Gronwall's inequality together with \eqref{12.1} yields \eqref{10.3}.
\end{proof}
We derive the following estimate on the $L^\infty(0,t;L^2)-$ norm of $\nabla u,$ which is  crucial to obtain the estimate of $\sup_{0\leq t\leq
T}\|\sqrt{\rho}\dot{u}||_{L^2}.$
\begin{lemma}\label{ZZY} Suppose that \eqref{12.1} is valid. Then there holds
\begin{equation}\begin{aligned}\label{12.2}
\sup_{0\leq t\leq T}(\|\nabla u||^2_{L^2}+||P||^2_{L^2})+\int_{0}^{T}\int_{\Omega}\rho|\dot{u}|^2dxdt\leq C,
\end{aligned}\end{equation}
for any ~$T\in[0,T^{\ast})$.~
\end{lemma}
\begin{proof}
It follows from $\eqref{FNSZ}_2$ that  the effective viscous flux $F$ solves a Neumann problem as follows:
\begin{equation}\begin{cases}\label{LAF}
\triangle F=\text{div}(\rho\dot{u})                     &\quad\text{in}  ~~\Omega,\\
\frac{\partial F}{\partial n}=\rho\dot{u}\cdot{n}        &\quad\text{on} ~~\partial\Omega.
\end{cases}\end{equation}
Meanwhile, $w$ solves the related Dirichlet problem:
\begin{equation}\begin{cases}\label{FAL}
\mu\triangle w=\nabla^{\perp}\cdot{(\rho\dot{u})}~~&\quad\text{in}~~\Omega,\\
w=0~~&\quad\text{on}~~\partial \Omega.
\end{cases}\end{equation}
\par
Then, applying the standard~$L^p$~elliptic regularity theory to~\eqref{LAF}~and~\eqref{FAL}, for~$k\geq0$~and~$p\in(1,\infty),$ we get
\begin{equation}\begin{aligned}
\|\nabla F||_{W^{k,p}}+||\nabla w||_{W^{k,p}}\leq C(p,k)\|\rho\dot{u}||_{W^{k,p}}.
\end{aligned}\end{equation}
In particular,
\begin{equation}\begin{aligned}\label{18.1}
\|\nabla F||_{L^2}+||\nabla w||_{L^2}\leq C\|\rho\dot{u}||_{L^2}.
\end{aligned}\end{equation}
Thus, for~$p\geq2,$~by virtue of \eqref{11.3}, Gagliardo-Nirenberg inequality and \eqref{18.1}, it follows that
\begin{equation}\begin{aligned}\label{184}
\|\nabla u||_{L^p}&\leq C(\|\text{div}u||_{L^p}+||w||_{L^p})\\
&\leq C(||F||_{L^p}+||w||_{L^p}+||P||_{L^p})\\
&\leq C(||F||_{L^2}^{\frac{2}{p}}\|\nabla F||^{1-\frac{2}{p}}_{L^2}+||F||_{L^2}+||w||^{\frac{2}{p}}_{L^2}\|\nabla w||^{1-\frac{2}{p}}_{L^2}+||P||_{L^p})\\
&\leq C\|\rho\dot{u}||^{1-\frac{2}{p}}_{L^2}(\|\nabla u||_{L^2}+||P||_{L^2})^{\frac{2}{p}}+C(\|\nabla u||_{L^2}+||P||_{L^p}).\\
\end{aligned}\end{equation}
Direct calculations show that
\begin{equation}\begin{aligned}
\nabla^{\perp}\cdot{\dot{u}}=&\frac{D}{Dt}w-(\partial_1u_1\partial_1u_2+\partial_1u_2\partial_2u_2)+\partial_2u_1\partial_1u_1+\partial_2u_2\partial_2u_1\\
=&\frac{D}{Dt}w+w\text{div}u,
\end{aligned}\end{equation}
and
\begin{equation}\begin{aligned}
\text{div}\dot{u}=&\frac{D}{Dt}\text{div}u+(\partial_1 u\cdot{\nabla})u_1+(\partial_2 u\cdot{\nabla})u_2\\
=&\frac{1}{2\mu+\lambda}\frac{D}{Dt}(F+P)+(\text{div}u)^2-2\nabla u_1\cdot{\nabla^{\perp}}u_2.
\end{aligned}\end{equation}
We rewrite the momentum equation as
\begin{equation}\begin{aligned}\label{12.23}
\rho\dot{u}=\nabla F+\mu\nabla^{\perp}w.
\end{aligned}\end{equation}
Multiplying $\eqref{12.23}$ by $\dot{u}$ and integrating over $\Omega$, thanks to \eqref{18.2} and the boundary conition \eqref{666}, we have
\begin{equation}\begin{aligned}\label{12.22}
&\frac{1}{2(2\mu+\lambda)}\frac{d}{dt}\int_{\Omega}F^2dx+\frac{\mu}{2}\frac{d}{dt}\int_{\Omega}w^2dx+\int_{\Omega}\rho|\dot{u}|^2dx\\
&=-\frac{\mu}{2}\int_{\Omega}w^2\text{div}udx-\int_{\Omega}F(\text{div}u)^2dx\\
&+2\int_{\Omega}F\nabla u_1\cdot{\nabla^{\perp}}u_2dx+\frac{1}{2(2\mu+\lambda)}\int_{\Omega}F^2\text{div}udx+\frac{2}{2\mu+\lambda}\int_{\Omega}FP\text{div}udx\\
&-\frac{1}{2\mu+\lambda}\int_{\Omega}F(\frac{\mu}{2}|\nabla u+(\nabla
u)^{tr}|^2+\lambda(\text{div}u)^2)dx+\int_{\partial\Omega}Fu\cdot{\nabla u}\cdot{n}ds\\
&\leq C\int_{\Omega}w^2\text{div}udx+C\int_{\Omega}F^2\text{div}udx\\
&+C\int_{\Omega}FP\text{div}udx+C\int_{\Omega}F|\nabla u|^2dx+C\int_{\partial\Omega}Fu\cdot{\nabla u}\cdot{n}ds\\
&=\sum_{i=1}^{5}H_i.
\end{aligned}\end{equation}
Now we estimate $H_i(i=1, 2, 3, 4, 5)$ respectively.

Together with Gagliardo-Nirenberg inequality and \eqref{18.1}, for~$\epsilon>0$~to be determined later, it yields that
\begin{equation}\begin{aligned}\label{12.31}
|H_1|=\left|\int_{\Omega}w^2\text{div}udx\right|\leq&C||w||^2_{L^4}\|\nabla u||_{L^2}\\
\leq&C||w||_{L^2}\|\nabla w||_{L^2}\|\nabla u||_{L^2}\\
\leq&\epsilon\|\sqrt{\rho}\dot{u}||^2_{L^2}+C(\epsilon)\|\nabla u||^4_{L^2}.
\end{aligned}\end{equation}
For~$\epsilon>0,$~it follows from H\"older inequality, Gagliardo-Nirenberg inequality and \eqref{18.1}, \eqref{10.3}  that
\begin{equation}\begin{aligned}
|H_2|=&\left|\int_{\Omega}F^2\text{div}udx\right|\leq C||F||^2_{L^4}\|\nabla u||_{L^2}\\
\leq&C\left(||F||_{L^2}\|\nabla F||_{L^2}+||F||^2_{L^2}\right)\|\nabla u||_{L^2}\\
\leq&C(||P||_{L^2}+\|\nabla u||_{L^2})\|\rho\dot{u}||_{L^2}\|\nabla u||_{L^2}+C(||P||_{L^2}+\|\nabla u||_{L^2})^2\|\nabla u||_{L^2}\\
\leq&\epsilon\|\sqrt{\rho}\dot{u}||^2_{L^2}+C(\epsilon)||P||^4_{L^2}+C(\epsilon)\|\nabla u||^4_{L^2}\\
\leq&\epsilon\|\sqrt{\rho}\dot{u}||^2_{L^2}+C(\epsilon)||P||_{L^1}||P||^3_{L^3}+C(\epsilon)\|\nabla u||^4_{L^2}\\
\leq&\epsilon\|\sqrt{\rho}\dot{u}||^2_{L^2}+C(\epsilon)\|\nabla u||^4_{L^2}+C(\epsilon)||P||^3_{L^3}.
\end{aligned}\end{equation}
Then, for~$\epsilon>0,$ using H\"older inequality, Gagliardo-Nirenberg inequality and \eqref{18.1}, \eqref{10.3} lead to
\begin{equation}\begin{aligned}
|H_3|=&\left|\int_{\Omega}FP\text{div}udx\right|\leq||P||_{L^3}\|\nabla u||_{L^2}||F||_{L^6}\\
\leq&C||P||_{L^3}\|\nabla u||_{L^2}(||F||^{\frac{1}{3}}_{L^2}\|\nabla F||^{\frac{2}{3}}_{L^2}+||F||_{L^2})\\
\leq&C||P||_{L^3}\|\nabla u||_{L^2}\left((||P||_{L^2}+\|\nabla u||_{L^2})^{\frac{1}{3}}\|\rho\dot{u}||^{\frac{2}{3}}_{L^2}+C||P||_{L^2}+\|\nabla
u||_{L^2}\right)\\
\leq&\epsilon\|\sqrt{\rho}\dot{u}||^2_{L^2}+C(\epsilon)\|\nabla u||^4_{L^2}+C(\epsilon)||P||^4_{L^2}+C(\epsilon)||P||^3_{L^3}\\
\leq&\epsilon\|\sqrt{\rho}\dot{u}||^2_{L^2}+C(\epsilon)\|\nabla u||^4_{L^2}+C(\epsilon)||P||_{L^1}||P||^3_{L^3}+C(\epsilon)||P||^3_{L^3}\\
\leq&\epsilon\|\sqrt{\rho}\dot{u}||^2_{L^2}+C(\epsilon)\|\nabla u||^4_{L^2}+C(\epsilon)||P||^3_{L^3}.
\end{aligned}\end{equation}
Use H\"older inequality, Gagliardo-Nirenberg inequality, and \eqref{18.1} and \eqref{184} to obtain
\begin{equation}\begin{aligned}
|H_4|=&\left|\int_{\Omega}F|\nabla u|^2dx\right|\leq C||F||_{L^3}\|\nabla u||^2_{L^3}\\
\leq&C\left(||F||^{\frac{2}{3}}_{L^2}\|\nabla F||^{\frac{1}{3}}_{L^2}+||F||_{L^2}\right)\B((||P||_{L^2}\\
&+\|\nabla u||_{L^2})^{\frac{4}{3}}\|\rho\dot{u}||^{\frac{2}{3}}_{L^2}+||P||^2_{L^3}+\|\nabla u||^2_{L^2}\B)\\
\leq&C(||P||_{L^2}+\|\nabla u||_{L^2})^2\|\rho\dot{u}||_{L^2}+C(||P||_{L^2}+\|\nabla
u||_{L^2})^{\frac{2}{3}}\|\rho\dot{u}||^{\frac{1}{3}}_{L^2}||P||^2_{L^3}\\
&+C\|\rho\dot{u}||^{\frac{2}{3}}_{L^2}(||P||_{L^2}+\|\nabla u||_{L^2})^{\frac{7}{3}}+C(||P||_{L^2}+\|\nabla u||_{L^2})||P||^2_{L^3}\\
&+C(||P||_{L^2}+\|\nabla u||_{L^2})^{\frac{8}{3}}_{L^2}\|\rho\dot{u}||^{\frac{1}{3}}_{L^2}+C(||P||_{L^2}+\|\nabla u||_{L^2})^3\\
\leq&\epsilon\|\sqrt{\rho}\dot{u}||^2_{L^2}+C(\epsilon)||P||^4_{L^2}+C(\epsilon)\|\nabla u||^4_{L^2}+C(\epsilon)||P||^3_{L^3}\\
\leq&\epsilon\|\sqrt{\rho}\dot{u}||^2_{L^2}+C(\epsilon)||P||_{L^1}||P||^3_{L^3}+C(\epsilon)\|\nabla u||^4_{L^2}+C||P||^3_{L^3}\\
\leq&\epsilon\|\sqrt{\rho}\dot{u}||^2_{L^2}+C(\epsilon)\|\nabla u||^4_{L^2}+C(\epsilon)||P||^3_{L^3},
\end{aligned}\end{equation}
where ~$\epsilon>0.$~
\par
Due to \eqref{13.1} and trace theorem, for~$\epsilon>0,$~we note that
\begin{equation}\begin{aligned}\label{12.32}
|H_5|=\left|\int_{\partial\Omega}Fu\cdot{\nabla u}\cdot{n}ds\right|=&\left|\int_{\partial\Omega}Fu\cdot{\nabla n}\cdot{u}ds\right|\\
\leq&C||F||_{L^2(\partial\Omega)}||u||^2_{L^4(\partial\Omega)}\\
\leq&C||F||_{H^1(\Omega)}\|\nabla u||^2_{L^2(\Omega)}\\
\leq&\epsilon\|\sqrt{\rho}\dot{u}||^2_{L^2}+C(\epsilon)\|\nabla u||^4_{L^2}.
\end{aligned}\end{equation}
Together with \eqref{12.31}-\eqref{12.32}, and taking $\epsilon$ sufficiently small we finally get
\begin{equation}\begin{aligned}\label{2022}
&\frac{1}{2}\frac{d}{dt}\int_{\Omega}\left(\frac{1}{(2\mu+\lambda)}P^2+(2\mu+\lambda)(\text{div}u)^2-2P\text{div}u+\mu
w^2\right)dx+\int_{\Omega}\rho|\dot{u}|^2dx\\
\leq&C\|\nabla u||^4_{L^2}+C||P||^3_{L^3}\\
\leq&C\|\nabla u||^4_{L^2}+C||P||_{L^\infty}||P||^2_{L^2}\\
\leq&C(||P||_{L^\infty}+\|\nabla u||^2_{L^2})(\|\nabla u||^2_{L^2}+||P||^2_{L^2}).
\end{aligned}\end{equation}
Multiplying $\eqref{FNSZ}_3$ by $P,$ and using the boundary condition \eqref{666} and
integrating over $\Omega,$ we obtain after integration by parts that
\begin{equation}\begin{aligned}\label{2023}
\frac{1}{2}\frac{d}{dt}\int_{\Omega}P^2dx=&\int_{\Omega}(\frac{\mu}{2}\left|\nabla u+(\nabla
u)^{\text{tr}}\right|^2+\lambda(\text{div}u)^2-\frac{3}{2}P\text{div}u)Pdx\\
\leq&C||P||_{L^\infty}\left(\|\nabla u||^2_{L^2}+||P||^2_{L^2}\right).\\
\end{aligned}\end{equation}
Noting that
\begin{equation}\begin{aligned}
&\left|\int_{\Omega}P\text{div}udx\right|\leq \frac{\mu}{4}\int_{\Omega}(\text{div}u)^2dx+\frac{C_1}{2}\int_{\Omega}P^2dx.\\
\end{aligned}\end{equation}
Multiplying $(C_1+1)$ on both sides of \eqref{2023} and then adding  to \eqref{2022}, we get
\begin{equation}\begin{aligned}\label{530}
&\frac{1}{2}\frac{d}{dt}\int_{\Omega}\left(\frac{1}{(2\mu+\lambda)}P^2+(2\mu+\lambda)(\text{div}u)^2-2P\text{div}u+\mu
w^2+(C_1+1)P^2\right)dx\\
&+\int_{\Omega}\rho|\dot{u}|^2dx\leq C(||P||_{L^\infty}+\|\nabla u||^2_{L^2})(\|\nabla u||^2_{L^2}+||P||^2_{L^2}).
\end{aligned}\end{equation}
Note that
\begin{equation}\begin{aligned}\label{531}
&\int_{\Omega}\left(\mu w^2+(2\mu+\lambda)(\text{div}u)^2-2P\text{div}u+(C_1+1)P^2\right)dx\geq \int_{\Omega}(|\nabla u|^2+P^2)dx.
\end{aligned}\end{equation}
It follows from Gronwall's inequality \eqref{531}, \eqref{12.1} and \eqref{10.3} that
\begin{equation}\begin{aligned}
\sup_{0\leq t\leq T}(\|\nabla u||^2_{L^2}+||P||^2_{L^2})+\int_{0}^{T}\int_{\Omega}\rho|\dot{u}|^2dxdt\leq C.
\end{aligned}\end{equation}

The proof of Lemma \ref{ZZY} is complete.
\end{proof}

\begin{lemma}
For~$p\geq 1$, there exist positive constants ~$C_{2}(p,\Omega)$~and~$C_{3}(\Omega)$~such that
\begin{equation}\begin{aligned}\label{10.9}
\|\dot{u}||_{L^p}\leq C_2(\|\nabla\dot{u}||_{L^2}+\|\nabla u||^2_{L^2}),
\end{aligned}\end{equation}
\begin{equation}\begin{aligned}\label{1110}
\|\nabla\dot{u}||_{L^2}\leq C_3(\|\text{div\,}\dot{u}||_{L^2}+\|\text{curl\,}\dot{u}||^2_{L^2}+\|\nabla u||^2_{L^4}).
\end{aligned}\end{equation}
\end{lemma}
\begin{proof}For more details we refer the readers to \cite{[JM]} and we omit them here.
\end{proof}
In the following lemma, we will give the estimate on $\nabla\dot{u}.$
\begin{lemma}\label{BBA} Suppose that \eqref{12.1} is valid. Then there holds
\begin{equation}\begin{aligned}\label{201.3}
\sup_{0\leq t\leq T}\int_{\Omega}\left(\rho|\dot{u}|^2+P^4\right)dx+\int_{0}^{T}\int_{\Omega}|\nabla\dot{u}|^2dx\leq C,
\end{aligned}\end{equation}
for any ~$T\in[0,T^{\ast}).$~
\end{lemma}
\begin{proof}
Applying the operator~$\dot{u_j}[\partial/\partial t+\text{div}(u\cdot)]$~to~$\eqref{FNS}_j,$ and combining with the boundary condition
\eqref{666}, we obtain after integrating the resulting equation over $\Omega$,
\begin{equation}\begin{aligned}\label{12.5}
\frac{1}{2}\frac{d}{dt}\int_{\Omega}\rho|\dot{u}|^2dx&=\int_{\Omega}(\dot{u}\cdot{\nabla F_t}+\dot{u}_{j}\text{div}(\partial_{j} F u))dx\\
+&\mu\int_{\Omega}(\dot{u}\cdot{\nabla^{\perp} w_t}+\dot{u}_{j}\text{div}((\nabla^{\perp}w)_j u))dx\\
&\triangleq I_1+I_2.
\end{aligned}\end{equation}
Due to the boundary condition~\eqref{666},~it follows that
\begin{equation}\begin{aligned}\label{10.7}
I_1=&\int_{\Omega}(\dot{u}\cdot{\nabla F_t}+\dot{u}_{j}\text{div}(\partial_{j} F u))dx\\
=&\int_{\Omega}\left(\dot{u}\cdot{\nabla F_t}+\dot{u}_{j}\partial_j(u\cdot{\nabla F})+\dot{u}_{j}\partial_j F\text{div}u-\dot{u}\cdot{\nabla
u}\cdot{\nabla F}\right)dx\\
=&\int_{\partial\Omega}(F_t+u\cdot{\nabla F})({\dot{u}\cdot{n}})ds-\int_{\Omega}(F_t+u\cdot{\nabla F})\text{div}\dot{u}dx\\
+&\int_{\Omega}(\dot{u}_{j}\partial_j F\text{div}u-\dot{u}\cdot{\nabla u}\cdot{\nabla F})dx\\
=&\int_{\partial\Omega}F_t(\dot{u}\cdot{n})ds+\int_{\partial\Omega}u\cdot{\nabla
F}(\dot{u}\cdot{n})ds-(2\mu+\lambda)\int_{\Omega}(\text{div}\dot{u})^2dx\\
+&(2\mu+\lambda)\int_{\Omega}\text{div}\dot{u}\partial_iu_j\partial_ju_idx-2\int_{\Omega}P\text{div}u\text{div}\dot{u}dx\\
&+\int_{\Omega}\left(\frac{\mu}{2}\left|\nabla u+(\nabla u)^{\text{tr}}\right|^2+\lambda(\text{div}u)^2\right)\text{div}\dot{u}dx\\
&+\int_{\Omega}(\dot{u}_{j}\partial_j F\text{div}u-\dot{u}\cdot{\nabla u}\cdot{\nabla F})dx\\
\leq&\int_{\partial\Omega}F_t(\dot{u}\cdot{n})ds+\int_{\partial\Omega}u\cdot{\nabla
F}(\dot{u}\cdot{n})ds-\mu\int_{\Omega}(\text{div}\dot{u})^2dx+||P||^4_{L^4}\\
&+C\|\nabla u||^4_{L^4}+C\int_{\Omega}|\dot{u}||\nabla F||\nabla u|dx,
\end{aligned}\end{equation}
where in the third equality we have used
\begin{equation}\begin{aligned}
&F_t+u\cdot{\nabla F}=((2\mu+\lambda)\text{div}u-P)_{t}+u\cdot{\nabla((2\mu+\lambda)\text{div}u-P)}\\
=&(2\mu+\lambda)\text{div}u_t-P_t+(2\mu+\lambda)u\cdot{\nabla\text{div}u}-u\cdot{\nabla P}\\
=&(2\mu+\lambda)\text{div}(\dot{u}-u\cdot{\nabla u})+(2\mu+\lambda)u\cdot{\nabla\text{div}u}-(P_t+u\cdot{\nabla P})\\
=&(2\mu+\lambda)\text{div}\dot{u}-(2\mu+\lambda)\partial_i(u_j\partial_ju_i)+(2\mu+\lambda)u_j\partial_j\partial_iu_i\\
&+2P\text{div}u-\frac{\mu}{2}\left|\nabla u+(\nabla u)^{\text{tr}}\right|^2-\lambda(\text{div}u)^2\\
=&(2\mu+\lambda)\text{div}\dot{u}-(2\mu+\lambda)\partial_iu_j\partial_ju_i+2P\text{div}u-\frac{\mu}{2}\left|\nabla u+(\nabla
u)^{\text{tr}}\right|^2-\lambda(\text{div}u)^2.
\end{aligned}\end{equation}
To estimate the first boundary term on the right hand side of \eqref{10.7}, we note that
\begin{equation}\begin{aligned}\label{11.10}
&\int_{\partial\Omega}F_t(\dot{u}\cdot{n})ds\\
=&\int_{\partial\Omega}F_t(u\cdot{\nabla u}\cdot{n})ds\\
=&-\frac{d}{dt}\int_{\partial\Omega}F(u\cdot{\nabla n}\cdot{u})ds+\int_{\partial\Omega}F(u\cdot{\nabla n}\cdot{u})_{t}ds\\
=&-\frac{d}{dt}\int_{\partial\Omega}F(u\cdot{\nabla n}\cdot{u})ds+\int_{\partial\Omega}F(u_t\cdot{\nabla
n}\cdot{u})ds+\int_{\partial\Omega}F(u\cdot{\nabla n}\cdot{u_t})ds\\
=&-\frac{d}{dt}\int_{\partial\Omega}F(u\cdot{\nabla n}\cdot{u})ds+\left[\int_{\partial\Omega}F(\dot{u}\cdot{\nabla
n}\cdot{u})ds+\int_{\partial\Omega}F(u\cdot{\nabla n}\cdot{\dot{u}})ds\right]\\
&-\int_{\partial\Omega}F((u\cdot{\nabla u})\cdot{\nabla n}\cdot{u})ds-\int_{\partial\Omega}F(u\cdot{\nabla n}\cdot({u\cdot{\nabla u}}))ds\\
=&-\frac{d}{dt}\int_{\partial\Omega}F(u\cdot{\nabla n}\cdot{u})ds+Q_1+Q_2+Q_3.
\end{aligned}\end{equation}
Then, it follows \eqref{18.1}, \eqref{12.2}, \eqref{10.9} that
\begin{equation}\begin{aligned}\label{12.9}
Q_1=&\int_{\partial\Omega}F(\dot{u}\cdot{\nabla n}\cdot{u})ds+\int_{\partial\Omega}F(u\cdot{\nabla n}\cdot{\dot{u}})ds\\
\leq&C||u||_{H^1}\|\dot{u}||_{H^1}||F||_{H^1}\\
\leq&C\|\nabla u||_{L^2}(\|\nabla\dot{u}||_{L^2}+\|\nabla u||^2_{L^2})(||P||_{L^2}+\|\nabla u||_{L^2}+\|\sqrt{\rho}\dot{u}||_{L^2})\\
\leq&C(1+\|\nabla\dot{u}||_{L^2})(1+\|\sqrt{\rho}\dot{u}||_{L^2})\\
\leq&C\epsilon\|\nabla\dot{u}||^2_{L^2}+C(\epsilon)\|\sqrt{\rho}\dot{u}||^2_{L^2}+C(\epsilon).
\end{aligned}\end{equation}
From~\eqref{13.1}, \eqref{18.1}, we deduce that
\begin{equation}\begin{aligned}\label{12.8}
|Q_2|=&\left|\int_{\partial\Omega}F((u\cdot{\nabla u})\cdot{\nabla n}\cdot{u})ds\right|\\
=&\left|\int_{\partial\Omega}F(u\cdot{n^{\perp}})n^{\perp}\cdot{\nabla u_i\partial_i n_ju_j}ds\right|\\
=&\left|\int_{\Omega}\nabla^{\perp}\cdot{\left(\nabla u_{i}\partial_{i}n_{j}u_{j}F(u\cdot{n^{\perp}})\right)}dx\right|\\
=&\left|\int_{\Omega}\nabla u_{i}\cdot{\nabla^{\perp}\left(\partial_{i}n_{j}u_{j}F(u\cdot{n^{\perp}}\right)}dx\right|\\
\leq&C\int_{\Omega}|\nabla u|(|F||u|^2+|F||\nabla u||u|+|u|^2|\nabla F|)dx\\
\leq&C\|\nabla u||_{L^4}||F||_{L^4}||u||^2_{L^4}+\|\nabla u||^2_{L^4}||F||_{L^4}||u||_{L^4}+\|\nabla u||_{L^4}||u||^2_{L^8}\|\nabla F||_{L^2}\\
\leq&C\|\nabla u||^4_{L^4}+C||F||^2_{H^1}+C\\
\leq&C\|\nabla u||^4_{L^4}+C\|\rho\dot{u}||^2_{L^2}+C\|\nabla u||^2_{L^2}+C||P||^2_{L^2}\\
\leq&C\|\nabla u||^4_{L^4}+C\|\sqrt{\rho}\dot{u}||^2_{L^2}+C.\\
\end{aligned}\end{equation}
The term~$Q_3$~can be handled in a similar way
\begin{equation}\begin{aligned}\label{12.6}
|Q_3|\leq C\|\nabla u||^4_{L^4}+C\|\sqrt{\rho}\dot{u}||^2_{L^2}+C.
\end{aligned}\end{equation}
According to \eqref{12.9}-\eqref{12.6}, we have
\begin{equation}\begin{aligned}\label{19.6}
\int_{\partial\Omega}F_t(\dot{u}\cdot{n})ds\leq&-\frac{d}{dt}\int_{\partial\Omega}F(u\cdot{\nabla n}\cdot{u})ds+C\epsilon\|\nabla\dot{u}||^2_{L^2}\\
&+C(\epsilon)(\|\nabla u||^4_{L^4}+1+\|\sqrt{\rho}\dot{u}||^2_{L^2}).\\
\end{aligned}\end{equation}
For the second  boundary term  on the right hand side of \eqref{10.7}
\begin{equation}\begin{aligned}\label{778}
&\int_{\partial\Omega}(u\cdot{\nabla F})(\dot{u}\cdot{n})ds=\int_{\partial\Omega}(u\cdot{n^{\perp}})n^{\perp}\cdot{\nabla F(\dot{u}\cdot{n})}ds\\
=&\int_{\Omega}\nabla^{\perp}\cdot(({u\cdot{n^{\perp}})\nabla F(\dot{u}\cdot{n})})dx\\
=&\int_{\Omega}\nabla F\cdot{\nabla^{\perp}((u\cdot{n^{\perp}})(\dot{u}\cdot{n}))}dx\\
\leq&C\int_{\Omega}|\nabla F|\nabla u||\dot{u}|dx+C\int_{\Omega}|\nabla F||\nabla\dot{u}||u|dx\\
\leq&C\|\nabla F||_{L^2}\|\dot{u}||_{L^4}\|\nabla u||_{L^4}+C\|\nabla F||_{L^2}\|\nabla\dot{u}||_{L^2}||u||^{\frac{3}{4}}_{L^{12}}\|\nabla
u||^{\frac{1}{4}}_{L^4}\\
\leq&C\|\rho\dot{u}||_{L^2}(\|\nabla\dot{u}||_{L^2}+\|\nabla u||^2_{L^2})\|\nabla u||_{L^4}+C\|\rho\dot{u}||_{L^2}\|\nabla\dot{u}||_{L^2}\|\nabla
u||^{\frac{1}{4}}_{L^4}\\
\leq&C\epsilon\|\nabla\dot{u}||^2_{L^2}+C(\epsilon)\|\sqrt{\rho}\dot{u}||^2_{L^2}(1+\|\nabla u||^2_{L^4}).\\
\end{aligned}\end{equation}
Substituting~\eqref{778}~and \eqref{19.6} into \eqref{10.7} indicates that
\begin{equation}\begin{aligned}\label{18.99}
&I_1\leq -\mu\int_{\Omega}(\text{div}\dot{u})^2dx-\frac{d}{dt}\int_{\partial\Omega}F(u\cdot{\nabla n}\cdot{u})ds\\
+&C\epsilon\|\nabla\dot{u}||^2_{L^2}+C(\epsilon)\|\nabla u||^4_{L^4}+C(\epsilon)\|\sqrt{\rho}\dot{u}||^2_{L^2}(1+\|\nabla u||^2_{L^4})+C||P||^4_{L^4}+C(\epsilon).
\end{aligned}\end{equation}
It remains to estimate~$I_2$. Together with boundary condition~\eqref{666}~and the following fact,
\begin{equation}\begin{aligned}
\text{div}(uw)=\text{div}u\,w+u\cdot{\nabla w}=0,\quad\text{on}\quad{\partial\Omega,}
\end{aligned}\end{equation}
we get
\begin{equation}\begin{aligned}\label{12.3}
I_2=&\mu\int_{\Omega}(\dot{u}\cdot{\nabla^{\perp}w_t}+\dot{u_j}((\text{div}(\nabla^{\perp}w)_{j}u))dx\\
=&-\mu\int_{\Omega}\text{curl}\dot{u}w_{t}dx+\mu\int_{\Omega}\dot{u}_j\text{div}((\nabla^{\perp}w)_ju)dx\\
=&-\mu\int_{\Omega}(\text{curl}\dot{u})^2dx+\mu\int_{\Omega}\text{curl}\dot{u}\text{curl}(u\cdot{\nabla
u})dx-\mu\int_{\Omega}u\cdot{\nabla{\dot{u}}}\cdot({\nabla^{\perp}w})dx\\
=&-\mu\int_{\Omega}(\text{curl}\dot{u})^2dx+\mu\int_{\Omega}\text{curl}\dot{u}(\nabla^{\perp}u)^T:\nabla udx\\
+&\mu\int_{\Omega}\text{curl}\dot{u}u\cdot{\nabla w}dx
-\mu\int_{\Omega}u\cdot{\nabla{\dot{u}}}\cdot{\nabla^{\perp}w}dx\\
=&-\mu\int_{\Omega}(\text{curl}\dot{u})^2dx+\mu\int_{\Omega}\text{curl}\dot{u}(\nabla^{\perp}u)^T:\nabla udx\\
+&\mu\int_{\Omega}\text{curl}\dot{u}\,(\text{div}(u w)-\text{div}u\,w)dx+\mu\int_{\Omega}\nabla^{\perp}u\cdot{\nabla{\dot{u}}}wdx\\
&+\mu\int_{\Omega}u\cdot{\nabla\text{curl}\dot{u}}wdx\\
=&-\mu\int_{\Omega}(\text{curl}\dot{u})^2dx+\mu\int_{\Omega}\text{curl}\dot{u}(\nabla^{\perp}u)^T:\nabla udx\\
-&\mu\int_{\Omega}u\cdot{\nabla\text{curl}\dot{u}\,} w-\mu\int_{\Omega}\text{curl}\dot{u}\,\text{div}u\,wdx+\mu\int_{\Omega}\nabla^{\perp}u\cdot{\nabla{\dot{u}}}wdx\\
&+\mu\int_{\Omega}u\cdot{\nabla\text{curl}\dot{u}}wdx\\
=&-\mu\int_{\Omega}(\text{curl}\dot{u})^2dx+\mu\int_{\Omega}\text{curl}\dot{u}(\nabla^{\perp}u)^T:\nabla udx\\
-&\mu\int_{\Omega}\text{curl}\dot{u}\,\text{div}u\,wdx+\mu\int_{\Omega}\nabla^{\perp}u\cdot{\nabla\dot{u}}wdx\\
\leq&-\mu\int_{\Omega}(\text{curl}\dot{u})^2dx+C\epsilon\|\nabla\dot{u}||^2_{L^2}+C(\epsilon)\|\nabla u||^4_{L^4}.
\end{aligned}\end{equation}
Now we show the estimate of~$\|\nabla u||_{L^4}.$  According to \eqref{184}, \eqref{12.2}, we can easily get
\begin{equation}\begin{aligned}\label{991}
\|\nabla u||^4_{L^4}\leq&C\left(\|\rho\dot{u}||^2_{L^2}(\|\nabla u||^2_{L^2}+||P||^2_{L^2})+||P||^4_{L^4}+\|\nabla u||^4_{L^2}\right)\\
\leq&C\|\sqrt{\rho}\dot{u}||^2_{L^2}+C||P||^4_{L^4}+C.
\end{aligned}\end{equation}

It follows from \eqref{12.5},~\eqref{18.99}~and \eqref{12.3}, \eqref{991} that
\begin{equation}\begin{aligned}\label{788}
&\frac{1}{2}\frac{d}{dt}\int_{\Omega}\rho|\dot{u}|^2dx+\mu\int_{\Omega}(\text{div}\dot{u})^2dx+\mu\int_{\Omega}(\text{curl}\dot{u})^2dx\\
\leq&-\frac{d}{dt}\int_{\partial\Omega}F(u\cdot{\nabla n}\cdot{u})ds+C\epsilon\|\nabla\dot{u}||^2_{L^2}\\
&+C\|\sqrt{\rho}\dot{u}||^2_{L^2}(1+\|\sqrt{\rho}\dot{u}||_{L^2}+||P||^2_{L^4})+||P||^4_{L^4}.
\end{aligned}\end{equation}

It remains to estimate $\sup_{0\leq t\leq T}||P||_{L^4}.$ Multiplying $\eqref{FNSZ}_3$ by $4P^3,$ and integrating over $\Omega,$ together with \eqref{991}
and \eqref{666}, we obtain
\begin{equation}\begin{aligned}\label{2021}
\frac{d}{dt}\int_{\Omega}P^4dx=&4\int_{\Omega}\left(\frac{\mu}{2}\left|\nabla u+(\nabla
u)^{\text{tr}}\right|^2+\lambda(\text{div}u)^2-\frac{7}{4}P\text{div}u\right)P^3dx\\
\leq&C||P||_{L^\infty}\left(\|\nabla u||^4_{L^4}+||P||^4_{L^4}\right)\\
\leq&C||P||_{L^\infty}\left(\|\sqrt{\rho}\dot{u}||^2_{L^2}+||P||^4_{L^4}+1\right).
\end{aligned}\end{equation}
Combining \eqref{788}, \eqref{2021}, and choosing~$\epsilon$~suitably small yields
\begin{equation}\begin{aligned}\label{18.100}
&\frac{d}{dt}\int_{\Omega}\left(\rho|\dot{u}|^2+P^4\right)dx+\mu\int_{\Omega}|\nabla\dot{u}|^2dx\\
\leq&-\frac{d}{dt}\int_{\partial\Omega}F(u\cdot{\nabla
n}\cdot{u})ds+C(1+||P||_{L^\infty}+\|\sqrt{\rho}\dot{u}||^2_{L^2})\left(\|\sqrt{\rho}\dot{u}||^2_{L^2}+||P||^4_{L^4}+1\right).
\end{aligned}\end{equation}
To estimate the first term on the right hand side of \eqref{18.100}, we note that
\begin{equation}\begin{aligned}\label{19.144}
\left|\int_{\partial\Omega}F(u\cdot{\nabla u}\cdot{n})ds\right|=&\left|\int_{\partial\Omega}F(u\cdot{\nabla n}\cdot{u})ds\right|\\
\leq&C||F||_{L^2(\partial\Omega)}||u||^2_{L^4(\partial\Omega)}\\
\leq&C||F||_{H^1(\Omega)}\|\nabla u||^2_{L^2(\Omega)}\\
\leq&C(||F||_{L^2(\Omega)}+\|\nabla F||_{L^2(\Omega)})\\
\leq&C(1+\epsilon\|\sqrt\rho\dot{u}||^2_{L^2(\Omega)}).
\end{aligned}\end{equation}
due to \eqref{13.1}, \eqref{18.1} and \eqref{12.2}.
\par
Putting \eqref{19.144} into \eqref{18.100} and integrating the resulting equation over $[0,t]$ lead to
\begin{equation}\begin{aligned}
&\int_{\Omega}\left(\rho|\dot{u}|^2+P^4\right)dx+\mu\int_{0}^{t}\int_{\Omega}|\nabla\dot{u}|^2dx\\
&\leq(1+\epsilon\int_{\Omega}\rho|\dot{u}|^2dx)+\int_{0}^{t}(1+||P||_{L^\infty}+\|\sqrt{\rho}\dot{u}||^2_{L^2})\left(\|\sqrt{\rho}\dot{u}||^2_{L^2}+||P||^4_{L^4}\right).
\end{aligned}\end{equation}
Choosing $\epsilon$ suitably small, we obtain  that
\begin{equation}\begin{aligned}
&\int_{\Omega}\left(\rho|\dot{u}|^2+P^4\right)dx+\mu\int_{0}^{t}\int_{\Omega}|\nabla\dot{u}|^2dx\\
\leq&C+\int_{0}^{t}(1+||P||_{L^\infty}+\|\sqrt{\rho}\dot{u}||^2_{L^2})\left(\|\sqrt{\rho}\dot{u}||^2_{L^2}+||P||^4_{L^4}\right).
\end{aligned}\end{equation}
On account of~\eqref{12.1}, \eqref{12.2}~and Gronwall's inequality we have
\begin{equation}\begin{aligned}\label{19.19}
\sup_{0\leq t\leq T}\int_{\Omega}(\rho|\dot{u}|^2+P^4)dx+\int_{0}^{T}\int_{\Omega}|\nabla\dot{u}|^2dxdt\leq C.
\end{aligned}\end{equation}
The proof of Lemma \ref{BBA} is complete.
\end{proof}
\section{Improved regularity of the density and the pressure}
In this section, we will obtain improved regularity of the density and the pressure. To this end, we first state  the Beale-Kato-Majda type inequality which was first established in \cite{[ML]}.
\begin{lemma}\label{FW}\cite{{[JM]},{[ML]}}For ~$2<q<\infty$~,assume that ~$u\in W^{2,q}(\Omega)$~with~$u\cdot{n}=0$~and $\text{curl} u=0 $
on~$\partial\Omega.$~Then there is a constant C=C(q) such that the following estimate holds
\begin{equation}\begin{aligned}\label{11.4}
\|\nabla u||_{L^\infty}\leq C(\|\text{div}~u||_{L^\infty}+\|\text{curl}~u||_{L^\infty})\log(e+\|\nabla^2 u||_{L^q})+C\|\nabla u||_{L^2}+C.
\end{aligned}\end{equation}
\end{lemma}
Thanks to the estimates obtained in section 2, we will show further estimates on $\nabla\rho$ and $\nabla P$ which are needed to extend the local strong solution
to be a global one.
\begin{lemma}\label{LGZ}Suppose that \eqref{12.1} is valid. Then there holds that
\begin{equation}\begin{aligned}
\sup_{0\leq t\leq T}(\|\rho||_{W^{1,q}}+||P||_{W^{1,q}}+\|\nabla u||_{H^1})\leq C, \quad{q>2}
\end{aligned}\end{equation}
for any $0\in[0,T^*)$.
\end{lemma}
\begin{proof}
For $r\in[2,q]$, note that~$|\nabla\rho|^r$~satisfies
\begin{equation}\begin{aligned}\label{10.8}
&(|\nabla\rho|^r)_t+\text{div}(|\nabla\rho|^r u)+(r-1)|\nabla\rho|^r\text{div}u+r|\nabla\rho|^{r-2}(\nabla\rho)^{tr}\nabla u(\nabla\rho)\\
&+r\rho|\nabla\rho|^{r-2}\nabla\rho\cdot{\nabla\text{div}u}=0.
\end{aligned}\end{equation}
Integrating the resulting equation over~$\Omega,$ we obtain after integration by parts and using boundary condition~\eqref{666}~that
\begin{equation}\begin{aligned}\label{88.2}
\frac{d}{dt}\|\nabla\rho||_{L^r}\leq C(1+\|\nabla u||_{L^\infty})\|\nabla\rho||_{L^r}+C\|\nabla^2u||_{L^r}.
\end{aligned}\end{equation}
Applying $\nabla$ to $\eqref{FNSZ}_{3}$ and multiplying the resulting equation by $r|\nabla P|^{r-2}\nabla P,$ we have
\begin{equation}\begin{aligned}\label{542}
&(|\nabla P|^r)_t+\text{div}(|\nabla P|^r u)+(2r-1)|\nabla P|^r\text{div}u+r|\nabla P|^{r-2}(\nabla P)^{tr}\nabla u(\nabla P)\\
+&2rP|\nabla P|^{r-2}\nabla P\cdot{\nabla \text{div}u}-2r\mu|\nabla P|^{r-2}\nabla P\cdot{\nabla(\frac{\mu}{2}\left|\nabla u+(\nabla
u)^{\text{tr}}\right|^2)}\\
&-\lambda r|\nabla P|^{r-2}\nabla P\cdot{\nabla (\text{div}u)^2}=0.
\end{aligned}\end{equation}
Integrating \eqref{542} over $\Omega$ and  using boundary condition~\eqref{666}, we then get
\begin{equation}\begin{aligned}\label{1129}
\frac{d}{dt}\|\nabla P||_{L^r}\leq& C(||P||_{L^\infty}+\|\nabla u||_{L^\infty})(\|\nabla P||_{L^r}+\|\nabla^2 u||_{L^r}).\\
\end{aligned}\end{equation}
In fact, the standard $L^p$ estimate for elliptic system with boundary condition \eqref{666}, \eqref{18.1} yields
\begin{equation}\begin{aligned}\label{11.07}
\|\nabla^2 u||_{L^r}\leq& C(\|\nabla\text{div}u||_{L^r}+\|\nabla w||_{L^r})\\
\leq&C\left(\|\nabla F||_{L^r}+\|\nabla P||_{L^r}+\|\nabla w||_{L^r}\right)\\
\leq&C(\|\rho\dot{u}||_{L^r}+\|\nabla P||_{L^r}).\\
\end{aligned}\end{equation}
It follows from \eqref{10.9}, \eqref{201.3} that
\begin{equation}\begin{aligned}\label{2341}
\|\rho\dot{u}||_{L^q}\leq& C\|\rho\dot{u}||^{\frac{2(q-1)}{q^2-2}}_{L^2}\|\rho\dot{u}||^{\frac{q(q-2)}{q^2-2}}_{L^{q^2}}\\
\leq&C\|\rho\dot{u}||^{\frac{2(q-1)}{q^2-2}}_{L^2}\|\dot{u}||^{\frac{q(q-2)}{q^2-2}}_{L^{q^2}}\\
\leq&C\|\rho\dot{u}||^{\frac{2(q-1)}{q^2-2}}_{L^2}\|\dot{u}||^{\frac{q(q-2)}{q^2-2}}_{H^1}\\
\leq&C\|\rho\dot{u}||^{\frac{2(q-1)}{q^2-2}}_{L^2}\left(\|\dot{u}||^{\frac{q(q-2)}{q^2-2}}_{L^2}+\|\nabla\dot{u}||^{\frac{q(q-2)}{q^2-2}}_{L^2}\right)\\
\leq&C\|\rho\dot{u}||^{\frac{2(q-1)}{q^2-2}}_{L^2}\left(\|\nabla u||^{\frac{2q(q-2)}{q^2-2}}_{L^2}+\|\nabla\dot{u}||^{\frac{q(q-2)}{q^2-2}}_{L^2}\right)\\
\leq&C\|\rho\dot{u}||_{L^2}+C\|\rho\dot{u}||^{\frac{2(q-1)}{q^2-2}}_{L^2}\|\nabla\dot{u}||^{\frac{q(q-2)}{q^2-2}}_{L^2}+C\\
\leq&C\left(1+\|\nabla\dot{u}||^{\frac{q(q-2)}{q^2-2}}_{L^2}\right).
\end{aligned}\end{equation}
Next, one gets from the Gagliardo-Nirenberg inequality,~\eqref{18.1}, \eqref{2341} that
\begin{equation}\begin{aligned}\label{11.16}
&\|\text{div}u||_{L^\infty}+\|\text{curl} u||_{L^\infty}\\
\leq& C(||F||_{L^\infty}+||P||_{L^\infty})+C\|\text{curl} u||_{L^\infty}\\
\leq&C\left(||F||^{\frac{q-2}{2(q-1)}}_{L^2}\|\nabla F||^{\frac{q}{2(q-1)}}_{L^q}+||w||^{\frac{q-2}{2(q-1)}}_{L^2}\|\nabla
w||^{\frac{q}{2(q-1)}}_{L^q}+||P||_{L^\infty}+||F||_{L^2}\right)\\
\leq&C\left(1+\|\rho\dot{u}||^{\frac{q}{2(q-1)}}_{L^q}+||P||_{L^\infty}\right)\\
\leq&C\left(1+\|\nabla\dot{u}||^{\frac{q^2(q-2)}{2(q-1)(q^2-2)}}_{L^2}+||P||_{L^\infty}\right),
\end{aligned}\end{equation}
which together with Lemma \ref{FW}, \eqref{11.16}, \eqref{11.07} and \eqref{2341} yields that
\begin{equation}\begin{aligned}\label{345}
&\|\nabla u||_{L^\infty}\leq C(\|\text{div}u||_{L^\infty}+\|\text{curl} u||_{L^\infty})\log(e+\|\nabla^2 u||_{L^q})+C\|\nabla u||_{L^2}+C\\
&\leq
C\left(1+\|\nabla\dot{u}||^{\frac{q^2(q-2)}{2(q-1)(q^2-2)}}_{L^2}+||P||_{L^\infty}\right)\log\left(e+\|\nabla\dot{u}||^{\frac{q(q-2)}{q^2-2}}_{L^2}+\|\nabla
P||_{L^q}\right)+C\|\nabla u||_{L^2}+C.
\end{aligned}\end{equation}
Set~r=q in \eqref{1129}, together with \eqref{11.07}, \eqref{345}, we have
\begin{equation}\begin{aligned}\label{549}
&\frac{d}{dt}\|\nabla P||_{L^q}\\
\leq& C(||P||_{L^\infty}+\|\nabla u||_{L^\infty})(\|\nabla P||_{L^q}+\|\nabla^2 u||_{L^q})\\
\leq&C(\|\nabla u||_{L^\infty}+||P||_{L^\infty})(\|\nabla\dot{u}||^{\frac{2(q-1)}{q^2-2}}_{L^2}+\|\nabla P||_{L^q})\\
\leq&C\left(1+\|\nabla\dot{u}||^{\frac{q^2(q-2)}{2(q-1)(q^2-2)}}_{L^2}+||P||_{L^\infty}\right)\log\left(e+\|\nabla\dot{u}||^{\frac{q(q-2)}{q^2-2}}_{L^2}+\|\nabla
P||_{L^q}\right)\left(\|\nabla\dot{u}||^{\frac{q(q-2)}{q^2-2}}_{L^2}+\|\nabla P||_{L^q}+1\right).
\end{aligned}\end{equation}
Similarly, set r=q in \eqref{88.2}, we get
\begin{equation}\begin{aligned}\label{1128}
&\frac{d}{dt}\|\nabla\rho||_{L^q}\leq C(1+\|\nabla\dot{u}||_{L^2}+||P||_{L^\infty})\|\nabla\rho||_{L^{q}}\log(e+\|\nabla P||_{L^q}+\|\nabla
\dot{u}||_{L^2})\\
&+C\|\nabla\dot{u}||_{L^2}+C\|\nabla P||_{L^q}+C.
\end{aligned}\end{equation}
Combining with \eqref{549}, \eqref{1128} gives that
\begin{equation}\begin{aligned}\nonumber
\frac{d}{dt}(e+\|\nabla\rho||_{L^q}+\|\nabla P||_{L^q})\leq C(1+||P||^{p_0}_{L^\infty}+\|\nabla\dot{u}||^2_{L^2})\log(e+\|\nabla\rho||_{L^q}+\|\nabla
P||_{L^q})(\|\nabla\rho||_{L^q}+\|\nabla P||_{L^q}),
\end{aligned}\end{equation}
where $1<p_0\leq 2.$
\par
Let
$$
f(t)\triangleq e+\|\nabla\rho||_{L^q}+\|\nabla P||_{L^q},g(t)\triangleq 1+\|\nabla\dot{u}||^2_{L^2}+||P||^{p_0}_{L^\infty},
$$
where $1<p_0\leq 2,$
which yields that
$$
(\log f(t))'\leq Cg(t)\log f(t).
$$
Thus, it follows from  Gronwall's inequality \eqref{12.1} and~\eqref{201.3}~that
\begin{equation}\begin{aligned}\label{390}
\sup_{0\leq t\leq T}(\|\nabla\rho||_{L^q}+\|\nabla P||_{L^q})\leq C,
\end{aligned}\end{equation}
which combining with \eqref{345}, \eqref{390} gives that
\begin{equation}\begin{aligned}\label{391}
&\int_{0}^{T}\|\nabla u||_{L^\infty}dt\leq C\int_{0}^{T}(\|\nabla\dot{u}||^2_{L^2}+||P||^{p_0}_{L^\infty}+1)\leq C,
\end{aligned}\end{equation}
where $1<p_0\leq 2.$
\par
Taking $r=2,$ it thus follows from \eqref{88.2}, \eqref{1129}, \eqref{11.07}, \eqref{391}, \eqref{201.3} that
\begin{equation}\begin{aligned}\label{88.1}
\sup_{0\leq t\leq T}(\|\nabla\rho||_{L^2}+\|\nabla P||_{L^2})\leq C.
\end{aligned}\end{equation}
According to~\eqref{18.1},~\eqref{88.1}, \eqref{201.3}~we get
\begin{equation}\begin{aligned}\label{3.2}
\|\nabla^2 u||_{L^2}\leq&C(\|\nabla\text{div}u||_{L^2}+\|\nabla w||_{L^2})\\
\leq&C(\|\nabla F||_{L^2}+\|\nabla w||_{L^2}+\|\nabla\rho||_{L^2})\\
\leq&C(\|\sqrt{\rho}\dot{u}||_{L^2}+\|\nabla\rho||_{L^2})\\
\leq&C.
\end{aligned}\end{equation}
This completes the proof of Lemma \ref{LGZ}.
\end{proof}
\begin{proof}[Proof of  Theorem \ref{WAW}]
We could use $(\rho, u,P)\left(x, T^{*}\right) \triangleq \lim _{t \rightarrow T^{*}}(\rho, u,P)(x, t)$ as the
 initial data. Thus,~$(\rho, u, P)(x, T^*)$~satisfy compatibility conditions. With the estimates in Lemma \ref{QQW}-Lemma \ref{LGZ} and local existence theory, we can extend the local strong solutions beyond~$T^{*}.$
 This contradicts the definition of~$T^{*}.$~Thus, we complete the proof of  Theorem \ref{WAW}.
\end{proof}
\section*{Ackonwledgments}
 Q. Jiu is partially
supported by National Natural Sciences Foundation of China (No. 11931010, No. 12061003).

	\end{document}